\documentclass[a4paper,12pt]{amsart}

\usepackage{t1enc}
\def\frak{\mathfrak}
\def\lie#1{\mathfrak{#1}}

\def\Cal{\mathcal}

\def\bbC{\mathbb{C}}

\def\bbH{\mathbb{H}}

\def\bbR{\mathbb{R}}
\def\bbW{\mathbb{W}}
\def\bbX{\mathbb{X}}
\def\bbY{\mathbb{Y}}

\def\cD{\mathcal{D}}
\def\cE{\mathcal{E}}
\def\cF{\mathcal{F}}

\def\cT{\mathcal{T}}

\def\cC{\mathcal{C}}

\def\cS{\mathcal{S}}

\def\de{\delta}

\def\io{\iota}

\def\la{\lambda}
\def\rh{\rho}
\def\si{\sigma}

\def\Ga{\Gamma}
\def\De{\Delta}

\def\Si{\Sigma}
\def\Ph{\Phi}

\def\Up{\Upsilon}
\def\na{\nabla}

\newcommand{\cq}{{\Cal Q}}
\newcommand{\ce}{{\Cal E}}
\newcommand{\bg}{\mbox{\boldmath{$ g$}}}
\newcommand{\nd}{\nabla}
\newcommand{\Rho}{P}
\newcommand{\Ric}{\operatorname{Ric}}
\newcommand{\J}{J}      
\newcommand{\V}{P}          
\newcommand{\wh}{\widehat}
\newcommand{\wt}{\widetilde}
\newcommand{\ol}{\overline}

\def\form#1{\mathbf{#1}}

\newcommand{\id}{\operatorname{id}}
\newcommand{\End}{\operatorname{End}}

\newcommand{\Symb}{\operatorname{Symb}}

\newtheorem{theorem}{Theorem}[section]
\newtheorem{lemma}[theorem]{Lemma}
\newtheorem{proposition}[theorem]{Proposition}

\theoremstyle{remark}

\newtheorem*{remark*}{\rm\bf Remark}

\newtheorem{example}[theorem]{\rm\bf Example}

\usepackage{amssymb,stmaryrd}
\usepackage{amscd}

\newcommand{\nn}[1]{(\ref{#1})}

\def\sideremark#1{\ifvmode\leavevmode\fi\vadjust{\vbox to0pt{\vss
 \hbox to 0pt{\hskip\hsize\hskip1em
 \vbox{\hsize3cm\tiny\raggedright\pretolerance10000
 \noindent #1\hfill}\hss}\vbox to8pt{\vfil}\vss}}}%
                        
\def\idx#1{{\em #1\/}}

\author{Josef \v Silhan}
\title{Conformally invariant quantization -- towards complete classification}

\begin{document}

\begin{abstract}
Let $M$ be a smooth manifold equipped with a conformal structure,
$\cE[w]$ the space of densities with the the conformal weight $w$ and 
$\cD_{w,w+\de}$ the space of differential operators from 
$\cE[w]$ to $\cE[w+\de]$. Conformal quantization $Q$ is a 
right inverse of the principle symbol map on $\cD_{w,w+\de}$ such that $Q$
is conformally invariant and exists for all $w$.
This is known to exists for generic values of $\de$. We give explicit
formulae for $Q$ for all $\delta$ out of the set of critical weights
$\Si$. We provide a simple description of this set and   
conjecture its minimality.
\end{abstract}

\address{Max Planck Institute for Mathematics \\
Vivatsgasse 7 \\
53111 Bonn \\
Germany} 
\email{silhan@math.muni.cz}


\thanks{The author gratefully acknowledges the support of the 
Erwin Schr\"odinger Institute (Vienna) and the Max-Planck-Institute
(Bonn).}

\maketitle

\pagestyle{myheadings}
\markboth{Josef \v Silhan}{Conformally invariant quantization -- towards 
complete classification}

\section{Introduction}

The notion of quantization originates in physics. Here we view it 
as quest for a correspondence between a space of differential 
operators and the corresponding space of symbols. More specifically,
consider the space $\cD_0$ of differential 
operators acting on smooth functions on a smooth manifold $M$ and the space of 
symbols $\cS_0$. Quantization is a map $Q_0: \cS_0 \to \cD_0$ such that
$\Symb \circ Q_0 = \id|_{\cS_0}$ where $\Symb: \cD_0 \to \cS_0$ is the 
principal symbol map. If $\Ph \in \cD_0$ of the order $k$ has the principal 
symbol 
$\si$ then $\Ph-Q_0(\si) \in \cD_0$ has the order $k-1$. Iterating this
we obtain the isomorphism of vector spaces 
$\bigoplus_{i=0}^k \cS_0^i \cong \cD_0^k$
where $\cS^i_0 = \Ga(\bigodot^i TM) \subseteq \cS_0$ and 
$\cD^k_0 \subseteq \cD_0$ is the space of operators of order at most $k$.
Here $\bigodot^k$ is the $k$th symmetric tensor product.
We shall use the notation $Q_0^\si := Q_0(\si)$.

There is no natural quantization on a $M$. On the other hand, 
e.g.\ a choice of a linear connection $\na$ on $M$ yields a prefered 
quantization in an obvious way: if $\si \in \cS^k_0$ and $f \in C^\infty(M)$,
we put $Q_0^\si(f) = \si(\na^{(k)} f)$ where $\na^{(k)} f$ is the symmetrized
$k$--fold covariant derivative. 
Therefore there is a canonical 
quantization on every pseudo-Riemannian manifold $M$. Motivated by this 
observation one can ask whether there is a natural quantization for less
rigid geometrical structures on $M$. 

In this article we study the case when the manifold $M$ is equipped
with a conformal structure. This was iniciated by Duval, Lecomte and Ovsienko 
\cite{DLOex}, see also \cite{LeOvPr} for the projective case.
The study of quantization for these (and related) structures has been very 
active in recent years, we refer to the survey \cite{Ma} and references 
therin for the state of art. 

The conformal structure on manifold $M$ is a class of pseudo--Riemannian 
metrics
$[g] = \{fg \mid f \in \bbC^\infty(M), f>0 \}$ on a manifold $M$.
The homegeneous model is the pseudosphere
$M=S^{p,q} := S^p \times S^q$, where $(p,q)$ is the signature of $g$,
the product of the standard metrics on $S^p$ and $S^q$. This is homogeneous
space for $G=SO_0(p+1,q+1)$ acting on $S^{p,q}$ by conformal motions of $[g]$
and we have the isomorphism $S^{p,q} \cong G/P$ where $P \subseteq G$ is the 
Poincare conformal group of motions fixing a point, see \cite{CSbook}
for details. Then both $\cS_0$ and $\cD_0$ are 
$G$--modules and the question of conformally invariant quantization means 
to construct $Q_0: \cS_0 \to \cD_0$
which intertwines these $G$--actions. 
If we pass from $S^{p,q}$ to $\bbR^{p,q}$ via the stereographic projection, we
replace the $G$--action (which is not defined on $\bbR^{p,q}$) by the 
infinitesimal $\frak{g}$--action. The Lie algebra $\frak{g}$
of $G$ can be realized as a Lie algebra of (polynomial) vector fields on 
$\bbR^{p,q}$ and they act by the Lie derivative as infinitesimal conformal 
symmetries. 
The same can be done for every locally conformally flat manifold and
the invariance of $Q_0$ is given by this $\frak{g}$--action. This setting
is often taken as the starting point in the study of invariant 
(or equivariant) quantization \cite{DLOex}. It is natural to consider more 
generally bundles of conformal densities $E[w]$, $w \in \bbR$ 
(instead just functions) and the 
space of differential operators $\Ga(E[w_1]) \to \Ga(E[w_2])$
denoted by $\cD_{w_1,w_2}$.
Denoting by $\cD^k_{w_1,w_2}$ the space of operators of degree $\leq k$, 
the corresponding bundle of $k$th degree symbols is then 
$S_\de^k = (\bigodot^k TM) \otimes E[\de] \cong 
\cD^k_{w_1,w_2}/\cD^{k-1}_{w_1,w_2}$ where $\de = w_2-w_1$.
Note this is the notation used in the conformally invariant calculus; 
the space of densities can be also defined as  
$\cF_\la = \Ga(\otimes^\la(\bigwedge^n T^*M))$ where 
$\bigwedge^n T^*M \to M$ is the determinant bundle, $n=\dim(M)$.
Then one has the relation $\Ga(E[-nw]) = \cF_{w}$.

Summarizing, the question in the conformally flat case 
is whether for a given $\de \in \bbR$ there
is an isomorphism of $\lie{so}_{p+1,q+1}$-modules
\begin{equation} \label{first}
Q_\de: \cS_\de \longrightarrow \cD_{w,w+\de} 
\end{equation}
for all $w \in \bbR$ where $S_\de = (\bigodot TM) \otimes E[\de]$.
That is, the corresponding bilinear differential operator
$Q_\de: \cS_\de \times \cE[w] \to \cE[w+\de]$ is conformally invariant.
It turns out the answer is positive for a generic weight $\de$. 
More precisely, it is shown in \cite{DLOex}
that if $\de \not\in \wt{\Si}$ where $\wt{\Si}$ is the set of 
\idx{critical weights} from \cite{DLOex} then 
the conformal quantization $Q_\de$ exists.
Note to get a complete answer one needs to study critical weights 
for particular irreducible components of $\cS_\de$.

Now we turn to the curved case where $M$ is a manifold with the given 
conformal class $[g]$. Then there are generically no infinitesimal 
symmetries on $(M,[g])$ and by invariance of the quantization
$Q_\de: \cS_\de \to \cD_{w,w+\de}$ we mean the corresponding
bilinear operator $Q_\de: \cS_\de \times \cE[w] \to \cE[w+\de]$ is given
in terms of a Levi--Civita connection $\na$ from the conformal class, 
its curvature $R$ and algebraic operations in such 
a way that $Q_\de$ does not depend on the choice of $\na$.
(This is equivalent to the $\frak{so}_{p+1,q+1}$-invariance
on conformally flat manifolds \cite{CSbook}.)
Using the terminology of conformal geometry,
$Q_\de$ has a \idx{curved analogue}.  Note there is generally no hope
for uniqueness of $Q_\de$ as the curvature can modify conformal operators
in various ways. 

\vspace{1ex}

Let us briefly summarize the development iniciated by \cite{DLOex} where 
the conformally flat case is considered. On one hand, there are 
several results for lower order cases \cite{DuOvHam,Dj3,MaRaNc}.
On the other hand, in the recent Kroeske's thesis \cite{KrTh}, a general 
problem of construction of conformal bilinear operators 
$V_1 \times V_2 \to W$ for given 
irreducible conformal bundles $V_1$, $V_2$ and $W$ is solved provided 
conformal weights of the bundles concerned are not critical. In fact, 
the result in \cite{KrTh} is much stronger as it provides such construction
for the wide class of \idx{parabolic geometries}. Conformal geometry is the 
most studied parabolic structure, other parabolic geometries are e.g.\ 
projective, contact projective or CR. In particular, parabolic geometries
cover all ``IFFT--cases'' \cite{IFFT}.  
For conformal structures,  the case $Q_0$ is related
to construction of symmetries of differential operators, 
see e.g.\ \cite{Ea,EaLe} for the Laplace operator.
The construction in \cite{KrTh} is very general however it is clear how to
obtain quantization from the machinery developed there. (The question of 
symbols and possible dependence on $w$ is not explicitly addressed there).

To classify the conformal quantization, the two basic quastions are the 
minimality of the critical set $\Si$ in the flat case and existence 
(and explicit construction) of $Q_\de$ for $\de \not \in \Si$
in the curved case. There are (up to our knowledge) no nonexistence 
results for critical conformal cases on $S^{p,q}$ hence the minimality is 
an issue. ($\wt{\Si}$ from \cite{DLOex}
is not minimal as observed in \cite{Dj3} for the third order quantization.)
An explicit construction for curved conformal manifolds is known only 
trace--free symbols in $\cS_\de$ \cite{RaT-free}.

Here we focus on the construction but also obtain a partial step towards
minimality of the critical set. The main result is
Theorem \ref{quant} which provides an explicit (and inductive) formula for
$Q_\de$ on all curved conformal manifolds. 
We obtain the critical set $\Si$ which is smaller than corresponding sets
in \cite{DLOex} or \cite{KrTh} and agrees with \cite{Dj3} for the 
order three. Moreover, we indicate some reasons why our set of critical weights 
$\Si$ should be minimal in Proposition \ref{critop}. We shall discuss 
minimality of this set in the follow up work \cite{SiNex} in detail. 

Let us comment upon what we mean by explicit construction. There is 
obviously no reason to ask for a formula in terms of a Levi--Civita
connection $\na$ (and its curvature) from the conformal class. These are
getting extremely complicated already for higher order \idx{linear} 
conformal operators
\cite{GoPetLap}. The conformal analogue of Riemannian $\na$--calculus is 
the \idx{tractor calculus},
see \cite{BEGo} for a discussion on its origin. 
It is closely related to the Cartan 
connection \cite{CapGotrans,luminy} and can be viewed as a linear or 
``explicit'' version of the Cartan connection. The transformation from 
tractors to formulae in terms of Levi--Civita connection 
is given by simple rules, see \cite{GoPetLap}
for a computer implementation.
In Theorem \ref{quant} we obtain simple tractor formulae for the conformal 
quantization $Q_\de$.
Then we discuss the critical set $\Si$ in details and conjecture its 
minimality, see Section \ref{scrit}.


The author would like to thank Andreas \v{C}ap for many fruitful 
discussions during this research.

\section{Conformal geometry and tractor calculus}

\subsection{Notation and background.}
We present here a brief summary, further details may
be found in \cite{CapGoamb,GoPetLap}.  Let $M$ be a smooth manifold of
dimension $n\geq 3$. Recall that a {\em conformal structure\/} of
signature $(p,q)$ on $M$ is a smooth ray subbundle $\cq\subset
S^2T^*M$ whose fibre over $x$ consists of conformally related
signature-$(p,q)$ metrics at the point $x$. Sections of $\cq$ are
metrics $g$ on $M$. So we may equivalently view the conformal
structure as the equivalence class $[g]$ of these conformally related
metrics.  The principal bundle $\pi:\cq\to M$ has structure group
$\bbR_+$, and so each representation ${\bbR}_+ \ni x\mapsto x^{-w/2}\in
{\rm End}(\bbR)$ induces a natural line bundle on $ (M,[g])$ that we
term the conformal density bundle $E[w]$. We shall write $ \ce[w]$ for
the space of sections of this bundle.  We write $\ce^a$ for the space
of sections of the tangent bundle $TM$ and $\ce_a$ for the space of
sections of $T^*M$. The indices here are abstract in the sense of
\cite{ot} and we follow the usual conventions from that source. So for
example $\ce_{ab}$ is the space of sections of $\otimes^2T^*M$.  Here
and throughout, sections, tensors, and functions are always smooth.
When no confusion is likely to arise, we will use the same notation
for a bundle and its section space.

We write $\bg$ for the {\em conformal metric}, that is the
tautological section of $S^2T^*M\otimes E[2]$ determined by the
conformal structure. This is used to identify $TM$ with
$T^*M[2]$.  For many calculations we employ abstract indices in an
obvious way.  Given a choice of metric $ g$ from $[g]$,
we write $ \nabla$ for the corresponding Levi-Civita connection. With
these conventions the Laplacian $ \Delta$ is given by
$\Delta=\bg^{ab}\nd_a\nd_b= \nd^b\nd_b\,$. Here we are raising indices
and contracting using the (inverse) conformal metric. Indices will be
raised and lowered in this way without further comment.  Note $E[w]$
is trivialized by a choice of metric $g$ from the conformal class, and
we also write $\nd$ for the connection corresponding to this
trivialization.  The coupled $ \nd_a$
preserves the conformal metric.

The curvature $R_{ab}{}^c{}_d$ of the Levi-Civita connection (the
Riemannian curvature) is given by $ [\nd_a,\nd_b]v^c=R_{ab}{}^c{}_dv^d
$ ($[\cdot,\cdot]$ indicates the commutator bracket).  This can be
decomposed into the totally trace-free Weyl curvature $C_{abcd}$ and a
remaining part described by the symmetric {\em Schouten tensor}
$\Rho_{ab}$, according to
\begin{equation}\label{csplit}
R_{abcd}=C_{abcd}+2\bg_{c[a}\Rho_{b]d}+2\bg_{d[b}\Rho_{a]c}, 
\end{equation}
 where
$[\cdots]$ indicates antisymmetrisation over the enclosed indices.
The Schouten tensor is a trace modification of the Ricci tensor
$\Ric_{ab}=R_{ca}{}^c{}_b$ 
and vice versa: $\Ric_{ab}=(n-2)\Rho_{ab}+\J\bg_{ab}$,
where we write $ \J$ for the trace $ \V_a{}^{a}$ of $ \V$.  The {\em
Cotton tensor} is defined by 
$
A_{abc}:=2\nabla_{[b}\Rho_{c]a} .
$
Via the Bianchi identity this is related to the divergence of the Weyl 
tensor as follows:
\begin{equation}\label{bi1} (n-3)A_{abc}=\nabla^d
C_{dabc} . 
\end{equation} 
Under a {\em conformal transformation} we replace a choice of metric 
$g$ by the metric $\hat{g}=e^{2\Up} g$, where $\Up$ 
 is a smooth
function. We recall that, in particular, the Weyl curvature is
conformally invariant $\widehat{C}_{abcd}=C_{abcd}$. 
With $\Up_a: = \na_a \Up$, the Schouten tensor transforms according to
\begin{equation}\label{Rhotrans}
\textstyle \widehat{\V}_{ab}=\V_{ab}-\nd_a \Up_b +\Up_a\Up_b
-\frac{1}{2} \Up^c\Up_c \bg_{ab} .
\end{equation}

Explicit formulae for the corresponding transformation of
the Levi-Civita connection and its curvatures are given in e.g.\ 
\cite{BEGo,GoPetLap}. From these, one can easily compute the 
transformation for a general valence (i.e.\ rank) $s$ section 
$f_{bc \cdots d} \in \cE_{bc \cdots d}[w]$ using the Leibniz rule:
\begin{equation} \label{grad_trans_gen}
\begin{split}
  \hat{\na}_{\bar{a}} f_{bc \cdots d}
  =&  \na_{\bar{a}} f_{bc \cdots d} + (w-s) \Up_{\bar{a}} f_{bc \cdots d}
     -\Up_{b} f_{\bar{a}c \cdots d} \cdots
     -\Up_{d} f_{bc \cdots \bar{a}} \\
   & +\Up^p f_{p c \cdots d} \bg_{b\bar{a}} \cdots
     +\Up^p f_{bc \cdots p} \bg_{d\bar{a}}.
\end{split}
\end{equation}

We next define the standard tractor bundle over $(M,[g])$.
It is a vector bundle of rank $n+2$ defined, for each $g\in[g]$,
by  $[\ce^A]_g=\ce[1]\oplus\ce_a[1]\oplus\ce[-1]$. 
If $\wh g=e^{2\Up}g$, we identify  
 $(\alpha,\mu_a,\tau)\in[\ce^A]_g$ with
$(\wh\alpha,\wh\mu_a,\wh\tau)\in[\ce^A]_{\wh g}$
by the transformation
\begin{equation}\label{transf-tractor}
 \begin{pmatrix}
 \wh\alpha\\ \wh\mu_a\\ \wh\tau
 \end{pmatrix}=
 \begin{pmatrix}
 1 & 0& 0\\
 \Up_a&\delta_a{}^b&0\\
- \tfrac{1}{2}\Up_c\Up^c &-\Up^b& 1
 \end{pmatrix} 
 \begin{pmatrix}
 \alpha\\ \mu_b\\ \tau
 \end{pmatrix} .
\end{equation}
It is straightforward to verify that these identifications are
consistent upon changing to a third metric from the conformal class,
and so taking the quotient by this equivalence relation defines the
{\em standard tractor bundle} $\ce^A$ over the conformal manifold.
(Alternatively the standard tractor bundle may be constructed as a
canonical quotient of a certain 2-jet bundle or as an associated
bundle to the normal conformal Cartan bundle \cite{luminy}.) On a
conformal structure of signature $(p,q)$, the bundle $\ce^A$ admits an
invariant metric $ h_{AB}$ of signature $(p+1,q+1)$ and an invariant
connection, which we shall also denote by $\nabla_a$, preserving
$h_{AB}$. Up up to isomorphism this the unique {\em normal conformal
tractor connection} \cite{CapGotrans} and it induces a normal
connection on $\bigotimes \ce^A$ that we will also denoted by $\na_a$
and term the (normal) tractor connection.  In a conformal scale $g$,
the metric $h_{AB}$ and $\na_a$ on $\ce^A$ are given by
\begin{equation}\label{basictrf}
 h_{AB}=\begin{pmatrix}
 0 & 0& 1\\
 0&\bg_{ab}&0\\
1 & 0 & 0
 \end{pmatrix}
\text{ and }
\nabla_a\begin{pmatrix}
 \alpha\\ \mu_b\\ \tau
 \end{pmatrix}
 =
\begin{pmatrix}
 \nabla_a \alpha-\mu_a \\
 \nabla_a \mu_b+ \bg_{ab} \tau +\Rho_{ab}\alpha \\
 \nabla_a \tau - \Rho_{ab}\mu^b  \end{pmatrix}. 
\end{equation}
It is readily verified that both of these are conformally well-defined,
i.e., independent of the choice of a metric $g\in [g]$.  Note that
$h_{AB}$ defines a section of $\ce_{AB}=\ce_A\otimes\ce_B$, where
$\ce_A$ is the dual bundle of $\ce^A$. Hence we may use $h_{AB}$ and
its inverse $h^{AB}$ to raise or lower indices of $\ce_A$, $\ce^A$ and
their tensor products.

In computations, it is often useful to introduce 
the `projectors' from $\ce^A$ to
the components $\ce[1]$, $\ce_a[1]$ and $\ce[-1]$ which are determined
by a choice of scale.
They are respectively denoted by $X_A\in\ce_A[1]$, 
$Z_{Aa}\in\ce_{Aa}[1]$ and $Y_A\in\ce_A[-1]$, where
 $\ce_{Aa}[w]=\ce_A\otimes\ce_a\otimes\ce[w]$, etc.
Using the metrics $h_{AB}$ and $\bg_{ab}$ to raise indices,
we define $X^A, Z^{Aa}, Y^A$. Then we
 see that 
\begin{equation} \label{trm}
  Y_AX^A=1,\ \ Z_{Ab}Z^A{}_c=\bg_{bc} ,
\end{equation}
and all other quadratic combinations that contract the tractor
index vanish. 
In \eqref{transf-tractor} note that  
$\wh{\alpha}=\alpha$ and hence $X^A$ is conformally invariant. 

The curvature $ \Omega$ of the tractor connection
is defined on $\ce^C$ by
$
[\nd_a,\nd_b] V^C= \Omega_{ab}{}^C{}_EV^E ~.
$
 Using
\eqref{basictrf} and the formulae for the Riemannian curvature yields
\begin{equation}\label{tractcurv}
\Omega_{abCE}= Z_C{}^cZ_E{}^e C_{abce}-2X_{[C}Z_{E]}{}^e A_{eab}
\end{equation}

Given a choice of $g\in [g]$,
the {\em tractor-$D$ operator}
$
D_A\colon\ce_{B \cdots E}[w]\to\ce_{AB\cdots E}[w-1]
$
is defined by 
\begin{equation}\label{Dform}
D_A V:=(n+2w-2)w Y_A V+ (n+2w-2)Z_{Aa}\nabla^a V -X_A\Box V, 
\end{equation} 
 where $\Box V :=\Delta V+w \J V$.  This is
 conformally invariant, as can be checked directly using the formulae
 above (or alternatively there are conformally invariant constructions
 of $D$, see e.g.\ \cite{Gosrni}).

The operator $D_A$ is strongly invariant. That is, it is invariant as 
an operator 
$$
D_A: \cE_{B \cdots E}[w]\to\ce_{AB\cdots E}[w-1]
$$
where now we interpret $\na$ in \nn{Dform} as the couple 
Levi--Civita--tractor connection. Note the 
strong invariance is a property of a \idx{formulae},
see \cite[p.21]{GoQsrni} for a more detailed discussion and 
\cite[(2)]{EaSrni95} for an illustrative example. 
We shall say an operator is strongly invariant if it is clear which
formula we mean. Note composition of two strongly invariant operators
is strongly invariant.

\subsection{Tractor connection and standard tractors}
Using the standard tractors $X_B$, $Z_B^b$ and $Y_B$, the tractor connections
takes the form
\begin{eqnarray} \label{XYZ}
\begin{split}  
  & \na_a Y_B \si = Y_B \na_a \si + Z_B^b P_{ab} \si, &&\si \in \cE[w] \\
  & \na_a Z_B^b \mu_b = -Y_B \mu_a + Z_B^b \na_a \mu_b 
    - X_B P_a{}^b \mu_b, &&\mu_b \in \cE_b[w] \\
  & \na_a X_B \rh = Z_B^b\bg_{ab} \rh + X_B \na_a \rh, &&\rh \in \cE[w]
\end{split}
\end{eqnarray}
which follows from \nn{basictrf} (or see e.g.\ \cite{GoPetLap}).
More accurately, $\na$ denotes the coupled tractor--Levi-Civita connection
in expressions like in the previous display.

We shall need, more generally, to know how the composition of several 
applications of the tractor connection acts on standard tractors. In fact, 
we shall need this only on $\bbR^{p,q}$. It follows from 
\nn{XYZ} (and can be verified easily by induction wrt.\ $k \geq 1$) that
\begin{align*}
  \na_{(a_1} \ldots \na_{a_k)} Y_B \si =& 
   \, Y_B \na_{(a_1} \ldots \na_{a_k)} \si + ct, \\
  \na_{(a_1} \ldots \na_{a_k)} Z_B^b \mu_b =& 
   - kY_B \de^b_{(a_1} \na_{a_2}^{} \ldots \na_{a_k)}^{} \mu_b^{}
   + Z_B^b \na_{(a_1} \ldots \na_{a_k)} \mu_b + ct, \\ 
  \na_{(a_1} \ldots \na_{a_k)} X_B \rh =& 
   - \frac{1}{2}k(k-1) Y_B \bg_{(a_1a_2} \na_{a_3} \ldots \na_{a_k)} \rh
   + kZ_B^b \bg_{b(a_1} \na_{a_2} \ldots \na_{a_k)} \rh  \\
  &+ X_B \na_{(a_1} \ldots \na_{a_k)} \rh + ct
\end{align*}
where $\si \in \cE[w]$, $\mu_b \in \cE_b[w]$, $\rh \in \cE[w]$ and 
``ct'' denotes terms which involve curvature and at most $k-2$ derivatives. 
(That is, ``ct'' vanishes on $\bbR^{p,q}$.) Here and below,
$(\ldots)$ denotes symmetrization of the enclosed indices and the notation
$(\ldots)_0$ will denote the projection to the symmetric trace-free part.
In fact, the previous display holds also for $k=0$ if we consider
expressions with $k$ free indices $a_1 \cdots a_k$ simply being absent
for $k=0$. Henceforth we shall use this convention. 
It follows from the previous display (or can be verified by induction 
directly) that for $k \geq 0$ we obtain
\begin{eqnarray} \label{nacomp}
\begin{split} 
  &&\na_{(a_1} \ldots \na_{a_k)_0} Y_B \si =& 
   \,Y_B \na_{(a_1} \ldots \na_{a_k)_0} \si + ct,  \\
  &&\qquad \na_{(a_1} \ldots \na_{a_k)_0} Z_B^b \mu_b =& 
   - kY_B \de^b_{(a_1} \na_{a_2} \ldots \na_{a_k)_0} \mu_b 
   + Z_B^b \na_{(a_1} \ldots \na_{a_k)_0} \mu_b + ct, \\
  &&\na_{(a_1} \ldots \na_{a_k)_0} X_B \rh =& 
   kZ_B^b \bg_{b(a_1} \na_{a_2} \ldots \na_{a_k)_0} \rh 
   + X_B \na_{(a_1} \ldots \na_{a_k)_0} \rh + ct
\end{split}
\end{eqnarray}
and for $\ell \geq 0$ we have
\begin{eqnarray} \label{Decomp}
\begin{split} 
  && \De^\ell Y_B \si =& \, Y_B \De^\ell \si + ct,  \\
  && \De^\ell Z_B^b \mu_b =& 
   - 2\ell Y_B \na^b \De^{\ell-1} \mu_b 
   + Z_B^b \De^\ell \mu_b + ct, \\
  && \De^\ell X_B \rh =& 
   - \ell(n+2\ell-2) Y_B \De^{\ell-1} \rh 
   + 2\ell Z_B^b \na_b \De^{\ell-1} \rh 
   + X_B \De^\ell \rh + ct
\end{split}
\end{eqnarray}
where $\si \in \cE[w]$, $\mu_b \in \cE_b[w]$, $\rh \in \cE[w]$.

\section{Tractor construction of conformal quantization and critical weights}

We assume 
$\si^{a_1 \ldots a_k} \in \cE^{(a_1 \ldots a_k)}[\de] =: \cS_{\de,k}$
and $f \in \cE[w]$. Our aim is to construct a quantization i.e.\
a differential operator $Q_\de^\si: \cE[w] \to \cE[w+\de]$
with the leading term
$\si^{a_1 \ldots a_k} \na_{a_1} \cdots \na_{a_k}$.
The bundle of symbols $\cE^{(a_1 \ldots a_k)}[\de]$ decomposes into 
irreducibles as
$$ 
\cE^{(a_1 \ldots a_k)}[\de] = \bigoplus_{i=0}^{\lfloor \frac{k}{2} \rfloor}
\cE^{(a_1 \ldots a_{k-2i})_0}[\de+2i]
$$
where $\lfloor a \rfloor$ denotes the lower integer part of $a \in \bbR$.
We can assume $\si$ is irreducible (as $Q_\de^\si$ is linear in $\si$) so
\begin{gather*}
\si^{a_1 \ldots a_k} = 
\si'{}^{(a_1 \ldots a_{k'}} \bg^{a_{k'+1}a_{k'+2}} \ldots 
\bg^{a_{k'+2\ell-1}a_{k'+2\ell})}, \ \ k'+2\ell=k
\quad \text{where} \\
(\si'){}^{a_1 \ldots a_{k'}} \in \cE^{(a_1 \cdots a_{k'})_0}[\de'], \ \
\de' = \de+2\ell
\end{gather*}
since $\bg^{ab} \in \cE^{ab}[-2]$. 

Henceforth we consider the irreducible symbol $\si'$ as in the 
previous display. Our aim is to construct a differential operator
\begin{eqnarray} \label{Q}
\begin{split}
  &Q_{k',\ell}^{\si'}: \cE[w] \to \cE[w+\de'-2l] \\
  &Q_{k',\ell}^{\si'} (f) = (\si')^{a_1 \ldots a_{k'}} 
   \na_{(a_1} \ldots \na_{a_{k'})_0} \De^\ell f + lot
\end{split}
\end{eqnarray}
which is conformally invariant as the bilinear operator
$Q_{k',\ell}: \cE^{(a_1 \cdots a_{k'})_0}[\de'] \times \cE[w] 
\to \cE[w+\de'-2l]$.
Here ``lot'' denotes lower order terms and we have suppressed
the parameter $\de'$ in the notation for $Q$. The reason is that 
we will define
the operator $Q_{k',\ell}^{\si'}: \cE[w] \to \cE[w+\de'-2l]$ by a universal 
tractor formula for all $\de' \in \bbR$. Then we shall discuss
when (i.e.\ for which $\de'$) $Q_{k',\ell}^{\si'}$ fails to have the 
required leading term.

The construction of $Q_{k',\ell}$ is divided into two steps -- the cases
$\ell=0$ and $\ell>0$.

\subsection{The quantization $Q_{k',0}$}
This case is more or less known. Here we shall formulate it as follows.

\begin{theorem} \label{quant0}
Let $(\si'){}^{a_1 \ldots a_{k'}} \in \cE^{(a_1 \ldots a_{k'})_0}[\de']$.
There is an explicit formula for the quantization 
$Q_{k',0}^{\si'}: \cE[w] \to \cE[w+\de']$ with the leading term
$(\si')^{a_1 \ldots a_{k'}}\na_{a_1} \ldots \na_{a_{k'}}$
for every weight $\de' \in \bbR$ satisfying
\begin{equation} \label{crit0}
  \de' \not\in \Si_{k',0} \quad \mbox{where} \quad 
  \Si_{k',0} =
  \begin{cases} \{ -(n+k'+i-2) \mid i = 1,\ldots,k' \} & k' \geq 1 \\
                \emptyset & k'=0.  
  \end{cases} 
\end{equation} 

Moreover, $Q_{k',0}$ is strongly conformally invariant in the following sense: 
if we replace $f \in \cE[w]$ by $f \in \cT \otimes \cE[w]$ for any tractor 
bundle $\cT$ and, in the formula for $Q_{k',0}$, we replace  
the Levi--Civita connection acting on $f$ by the coupled 
Levi--Civita--tractor connection then $Q_{k',0}$ is a conformally 
invariant bilinear operator 
$\cE^{(a_1 \ldots a_{k'})_0}[\de'] \times \cT \otimes \cE[w] \to
\cT \otimes \cE[w+\de']$.
\end{theorem}

\noindent
{\bf Remark.}
The conformal quantization for the case $Q_{k',0}$ was constructed
recently in \cite{RaT-free} 
but the strong invariance of this result is unclear.
To clarify this point and to keep our presentation
self--content, we present the complete proof here.

\begin{proof}
We shall use certain splitting operators from $\cE[w]$ and 
$\cE^{(a_1 \ldots a_{k'})_0}[\de']$ into symmetric tensor products of 
the adjoint tractor bundle $\cE_{[A^1A^2]}$ 
and their subquotients. To simplify the notation, we shall introduce
adjoint tractor indices $\form{A} := [A^1A^2]$. 
These are just abstract indices 
of the adjoint tractor bundle. We shall use the notation
$f_\form{(AB)} = \frac{1}{2}( f_\form{AB} + f_\form{BA})$,  
$f_\form{AB} \in \cE_\form{AB}$ for the symmetrization, the symmetric
tensor products of the adjoint tractor bundle will be denoted by
$\cE_\form{(A_1 \ldots A_k)}$. Let us note this notation means symmetrization
over adjoint indices (not \idx{not} over standard tractor indices), i.e.\
$f_\form{(AB)_0} \not= 0$. The completely trace free component
with respect to $h^{AB}$ will be denoted by $\cE_\form{(A_1 \ldots A_k)_0}$. 
Note the latter bundles are generally not irreducible tractor bundles.

The skew symmetrization with the tractor $X_{A_i^0}$ defines bundle maps
$\cE_\form{(A_1 \ldots A_k)_0} \to 
\cE_\form{A_1 \ldots [A_i^0A_i^{}] \ldots A_k}$. The joint kernel of
all these maps for $i=1,\ldots,k$ will be denoted by 
$\bar{\cE}_\form{(A_1 \ldots A_k)_0}$. Using the complement
$\bar{\cE}_\form{(A_1 \ldots A_k)_0}^\perp \subseteq 
\cE_\form{(A_1 \ldots A_k)_0}$ (via the tractor metric $h$), 
we obtain the quotient bundle 
$\tilde{\cE}_\form{(A_1 \ldots A_k)_0} := 
\cE_\form{(A_1 \ldots A_k)_0}/\bar{\cE}_\form{(A_1 \ldots A_k)_0}^\perp$.
One can easily see that choosing a metric from the conformal class,
sections of the these have the form
\begin{align*}
&\bar{F}_\form{(A_1 \ldots A_k)_0}  = 
 \sum_{i=0}^k \bbX_\form{A_1}^{\ a_1} \ldots \bbX_\form{A_i}^{\ a_i} 
 \bbW_\form{A_{i+1}}^{} \ldots \bbW_\form{A_k}^{} 
 \bar{f}^i_{(a_1 \ldots a_i)_0}
 \ \ \mbox{for} \ \ 
 \bar{F}_\form{(A_1 \ldots A_k)_0} \in \bar{\cE}_\form{(A_1 \ldots A_k)_0},
 \\
&\tilde{F}_\form{(A_1 \ldots A_k)_0}  = 
 \sum_{i=0}^k \bbY_\form{A_1}^{\ a_1} \ldots \bbY_\form{A_i}^{\ a_i} 
 \bbW_\form{A_{i+1}}^{} \ldots \bbW_\form{A_k}^{} 
 \tilde{f}^i_{(a_1 \ldots a_i)_0}
 \ \ \mbox{for} \ \ 
 \tilde{F}_\form{(A_1 \ldots A_k)_0} \in \tilde{\cE}_\form{(A_1 \ldots A_k)_0}
\end{align*}
for some sections $\bar{f}^i_{(a_1 \ldots a_i)_0}$ and 
$\tilde{f}^i_{(a_1 \ldots a_i)_0}$. Note $i$ is \idx{not} an abstract 
index here.
This describes the composition series 
for $\bar{\cE}_\form{(A_1 \ldots A_k)_0}$ and 
$\tilde{\cE}_\form{(A_1 \ldots A_k)_0}$. (In particular, choosing
a metric in the conformal class, both these bundles decompose to exactly 
$k+1$ irreducible components, e.g.\ 
$\bar{\cE}_\form{(A_1 \ldots A_k)_0} = \cE \oplus \cE_{a_1} \oplus \ldots
\oplus \cE_{(a_1 \ldots a_k)_0}$.)
Also note the latter bundle is the 
dual of the former one. Finally, taking the tensor product with
density bundles, we obtain
$\bar{\cE}_\form{(A_1 \ldots A_k)_0}[w]$ and 
$\tilde{\cE}_\form{(A_1 \ldots A_k)_0}[w]$ for any $w \in \bbR$.

Assume $k' \geq 1$. We shall start with
$f \in \cE[w]$. For an arbitrary chosen metric from the conformal
class, we consider the inclusion 
$$ 
\bar{\io}: \cE[w] \hookrightarrow 
\bar{\cE}_\form{(A_{{}_1} \ldots A_{k'})_0}[w], \quad 
f \stackrel{\bar{\io}}\mapsto \bbW_\form{A_{{}_1}} \cdots \bbW_\form{A_{k'}}f. 
$$
Now $\bar{\io}(f)$ can be extended to a conformally invariant section
$\bar{F}_\form{A_1 \ldots A_{k'}} \in
\bar{\cE}_\form{(A_1 \ldots A_{k'})_0}[w]$ as follows: we put
$\bar{F}_\form{A_1 \ldots A_{k'}} := 
\bar{P}(\cC) (\bar{\io}(f)_{A_1 \ldots A_{k'}})$ 
where the operator $\bar{P}(\cC)$ is a suitable polynomial in  
the \idx{curved Casimir}
$\cC: \bar{\cE}_\form{(A_1 \ldots A_{k'})_0}[w] \to
\bar{\cE}_\form{(A_1 \ldots A_{k'})_0}[w]$ \cite{CScC}. It follows
from the composition series for $\bar{\cE}_\form{(A_1 \ldots A_{k'})_0}[w]$
that the degree of the polynomial $\bar{P}$
is $k'$. Let us compute the highest order term of 
$\bar{P}(\cC) \bigl( \bbW_\form{A_1} \cdots \bbW_\form{A_{k'}}f \bigr)$.
For this it is sufficient to work on $\bbR^n$ with the standard metric.
Then if $P$ is a polynomial of degree $r$, $0 \leq r \leq k'$ then
there is a (degree $r$) polynomial $p$ such that
$P(\cC) \bigl( \bbW_\form{A_1} \cdots \bbW_\form{A_{k'}}f \bigr) =
\bbW_\form{A_1} \cdots \bbW_\form{A_{k'}} p(w) f + \ldots + 
\bbX_\form{(A_1}^{\ a_1} \ldots \bbX_\form{A_r}^{\ a_r}
\bbW_\form{A_{r+1}}^{} \cdots \bbW_\form{A_{k'})}^{} 
\na_{(a_1} \cdots \na_{a_r)_0} f$ up to a (nonzero) scalar multiple.
This can be easily verified by the induction. Putting $r:=k'$,
there is a $k'$--order polynomial $\bar{p}(w)$ such that 
\begin{equation} \label{split:f}
\bar{F}_\form{A_1 \ldots A_{k'}} = 
\bbW_\form{A_1} \cdots \bbW_\form{A_{k'}} \bar{p}(w) f + \ldots + 
\bbX_\form{(A_1}^{\ a_1} \ldots \bbX_\form{A_{k'})}^{\ a_{k'}}
\na_{(a_1} \cdots \na_{a_{k'})_0} f
\end{equation}
up to a nonzero scalar multiple.

The splitting for 
$(\si'){}^{a_1 \ldots a_{k'}} \in \cE^{(a_1 \ldots a_{k'})_0}[\de']$
is analogous. We shall start with the inclusion
$$ 
\tilde{\io}: \cE^{(a_1 \ldots a_{k'})_0}[\de'] \hookrightarrow 
\bar{\cE}_\form{(A_1 \ldots A_{k'})_0}[\de'],
\quad 
(\si'){}^{a_1 \ldots a_{k'}} \stackrel{\tilde{\io}}\mapsto 
\bbY_\form{A_1}^{\ a_1} \cdots \bbY_\form{A_{k'}}^{\ a_{k'}}
(\si'){}_{a_1 \ldots a_{k'}}
$$
for a chosen metric in the conformal class. 
Then we apply a suitable polynomial operator 
in the curved Casimir to obtain a conformally invariant extension
$\tilde{F}_\form{A_1 \ldots A_{k'}} := 
\tilde{P}(\cC) (\tilde{\io}(\si')_{A_1 \ldots A_k}) \in
\tilde{\cE}_\form{(A_1 \ldots A_k)_0}[\de']$.
A similar reasoning as above shows that $\tilde{P}$ has order $k'$ and
\begin{equation} \label{split:si}
\tilde{F}_\form{A_1 \ldots A_{k'}} = 
\bbY_{\form{A_1}a_1} \cdots \bbY_{\form{A_{k'}}a_{k'}} \tilde{p}(\de') 
(\si'){}^{a_1 \ldots a_{k'}} + \ldots + 
\bbW_\form{(A_1}\ldots \bbW_\form{A_{k'})}
\na_{(a_1} \cdots \na_{a_{k'})_0} (\si'){}^{a_1 \ldots a_{k'}}
\end{equation}
on $\bbR^n$ for a polynomial $\tilde{p}$ of the order $k'$. 
In this case we need to know
$\tilde{p}(\de')$ explicitly; following \cite{CScC} we computes
$$
\tilde{p}(\de') = \prod_{i=1}^{k'} (\de'+n+k'+i-2).
$$
In fact, analogues of this splitting are well--known, see e.g.\
\cite[6.2.3]{KrTh} or \cite[2.1.4]{SiTh}.

In the last step we use the duality between 
$\bar{\cE}_\form{(A_1 \ldots A_k)_0}$ and
$\tilde{\cE}_\form{(A_1 \ldots A_k)_0}$. From this it follows
that that 
$Q_{k',0}^{\si'} (f) := 
\tilde{F}^\form{A_1 \ldots A_{k'}} \bar{F}_\form{A_1 \ldots A_{k'}}$
is a conformally invariant bilinear operator. Considering
$Q_{k',0}^{\si'}$ as a linear operator $\cE[w] \to \cE[w+\de']$, 
it follows from \nn{split:f} and \nn{split:si} that
$$
\tilde{F}^\form{A_1 \ldots A_{k'}} \bar{F}_\form{A_1 \ldots A_{k'}}
= \tilde{p}(\de') (\si'){}^{a_1 \ldots a_{k'}} \na_{(a_1} \cdots \na_{a_{k'})_0} f
+ lot
$$
where ``$lot$'' denotes the lower order terms.

It remains to verify the strong invariance of $Q_{k',0}^{\si'}(f)$.
But this follows from the fact that the curved Casimir is a strongly
invariant linear differential operator.
\end{proof}

\begin{remark*}
The formula for the curved Casimir operator can be easily given explicitly
via tractors. First we define put $\bbH_{\form{AB}} := h_{A^1B^1}h_{A^2B^2}$
where we skew over $[A^1A^2]=\form{A}$ (hence also over $[B^1B^2]=\form{B}$).
Since $\cE_\form{B}$ is the adjoint tractor bundle, there is an inclusion
$\cE_\form{B} \hookrightarrow \End(\cT)$ for any tractor bundle $\cT$.
This yields also  $\cE_{\form{AB}} \hookrightarrow \cE_\form{A} \otimes
\End(\cT)$, the image of $\bbH_{\form{AB}}$ under this inclusion
 will be denoted by 
$\bbH_\form{A} \in \cE_\form{A} \otimes \End(\cT)$.
If $F \in \cT$, the application of this endomorphism will be denoted by
$\bbH_\form{A} \sharp F \in \cE_\form{A} \otimes \cT$. Explicitly,
$\bbH_\form{A} \sharp F_C = \bbH_{\form{A}C}{}^P F_P$ for $\cT = \cE_C$ and
the general case $\cT \subseteq (\bigotimes \cE_C) \otimes \cE[w]$ 
is given by the Leibnitz rule. (We put $\bbH_\form{A} \sharp$ to be trivial 
on $\cE[w]$.)

If $\cT$ is a tractor bundle then the differential operator
$$
\cD_\form{A}: \cT \otimes \cE[w] \to \cT \otimes \cE_\form{A}[w],
\quad
\cD_\form{A} := w\bbW_\form{A} + \bbX_{\form{A}}^{\,a} \na_a 
   + \bbH_\form{A} \sharp
$$
is the (conformally invariant) \idx{fundamental derivative} 
\cite{CapGotrans} up to a nonzero scalar multiple. 
The curved Casimir $\cC$ is defined as
$\cC := \cD^\form{A} \cD_\form{A}: \cT \otimes \cE[w] \to \cT \otimes \cE[w]$.
The explicit formula for $\cC$ in terms of a chosen Levi--Civita connection
from the conformal class can be easily obtained from the previous
display.
\end{remark*}

\subsection{The general case $Q_{k',\ell}$}
Recall $k',\ell \geq 0$,
$(\si')^{a_1 \ldots a_{k'}} \in \cE^{(a_1 \ldots a_{k'})_0}[\de']$
and $f \in \cE[w]$, $\de',w \in \bbR$.
We shall construct $Q_{k',\ell}$ by an inductive procedure.
The main step is the construction of $Q_{k',\ell+1}^{\si'}$ from 
$Q_{k',\ell}^{\si'}$. 

\begin{proposition} \label{ind}
Fix $\de' \in \bbR$ and assume
there is an explicit construction of the quantization 
$Q_{k',\ell}^{\si'}: \cE[w] \to \cE[w+\de'-2\ell]$, $k',\ell \geq 0$ 
with the leading term
$\si^{a_1 \ldots a_{k'}} \na_{a_1} \ldots \na_{a_{k'}} \De^\ell$
for every $w \in \bbR$. 
Also assume $Q_{k',\ell}$ is strongly invariant
in the sense of Theorem \ref{quant0}. Then
\begin{align*} 
&\wt{Q}_{k',l}^{\si'} := D^B Q_{k',\ell}^{\si'} D_B: 
  \cE[w] \to \cE[w+\de'-2(\ell+1)], \\
&\wt{Q}_{k',\ell}^{\si'}(f) = -(\de'-\ell)(n+2\de'+2(k'-\ell)-2)
  \si^{a_1 \ldots a_{k'}} \na_{a_1} \ldots \na_{a_{k'}} \De^{\ell+1}
  + lot
\end{align*}
for every $w \in \bbR$. Here ``$lot$'' denotes lower order terms.

The operator $\wt{Q}_{k',l}: \cE^{(a_1 \ldots a_{k'})_0}[\de'] \times \cE[w]
\to \cE[w+\de'-2(\ell+1)]$ is a conformally invariant bilinear operator.
Moreover, it is strongly invariant in the sense of Theorem \ref{quant0}.
We put $Q_{k',l+1}^{\si'} := \wt{Q}_{k',l}^{\si'}$.
\end{proposition}

\begin{proof}
We shall start with the discussion on the invariance. Since
$Q_{k',\ell}: \cE^{(a_1 \ldots a_{k'})_0}[\de'] \times \cE[w]
\to \cE[w+\de'-2\ell]$ is assumed to be strongly invariant (in the sense of 
Theorem \ref{quant0}), it is also
invariant as $Q_{k',l}: \cE^{(a_1 \ldots a_{k'})_0}[\de'] \times \cE_B[w]
\to \cE[w+\de'-2\ell]$. Therefore the composition 
$$ 
\cE^{(a_1 \ldots a_{k'})_0}[\de'] \times \cE[w] 
\stackrel{\id \times D_B}{\longrightarrow}
\cE^{(a_1 \ldots a_{k'})_0}[\de'] \times \cE_B[w-1]
\stackrel{Q_{k',\ell}}{\longrightarrow}
\cE[(w-1)+\de'-2\ell] \stackrel{D^B}{\longrightarrow}
\cE[w+\de'-2\ell-2] $$
is a conformally invariant bilinear operator. 
The strong invariance of $\wt{Q}_{k',l}$ follows from the strong 
invariance of $D^B$.

It remains to compute the leading symbol of $\wt{Q}_{k',\ell}^{\si'}$,
we shall do it by a direct
computation. The operator $D^B$ is explicitly given by the sum
of three terms on the right hand side of \nn{Dform}. Decomposing both 
application
of tractor $D$ in the formula for $\wt{Q}_{k',l}^{\si'}$ accordingly,
we obtain overall $9$ leading terms. 
Note $Q_{k',\ell}^{\si'} D_B f
= \bigl[ (\si')^{a_1 \ldots a_{k'}} \na_{a_1} \dots \na_{a_{k'}} + lot
\bigr] \De^\ell D_Bf \in \cE_B[w']$ 
where $f \in \cE[w]$ and $w' = w+\de'-2\ell-1$. 

Although the tractor $D$ 
is of the second order and $Q_{k',\ell}^{\si'}$ is of 
the order $k'+2\ell$, the leading term of $\wt{Q}_{k',\ell}^{\si'}$ 
turns out to have order $k'+2\ell+2$ in the generic case. 
(We use the tractor $D$ twice so one might expect the order $k'+2\ell+4$.)
To show this we will collect all terms of the order at least $k'+2\ell+2$. 
In fact, we shall do this in details only for the leading term 
$(\si')^{a_1 \ldots a_{k'}} \na_{a_1} \ldots \na_{a_{k'}} \De^\ell$ of 
$Q_{k',\ell}^{\si'}$.
But it will be obvious from the form of all 9 summands this is sufficient.
Below we shall use
$lot_{\leq o}$ to denote terms of the order at most $o$, 
$lot_{<o}$ will denotes cases of order smaller than $o$. 
To simplify the notation
we will henceforth work with the Euclidean metric; then all terms 
on the right hand side of \nn{nacomp} and \nn{Decomp} denoted by
``$ct$'' vanish.

We shall start with $w'(n+2w'-2)Y^B Q_{k',\ell}^{\si'} D_Bf$; 
decomposing $D_B$
here according to \nn{Dform} yields first three summands.
The first one is
\begin{eqnarray} \label{YY}
\begin{split}
  \quad\quad
  &w'(n+2w'-2) Y^B Q_{k',\ell}^{\si'} \bigl[ w(n+2w-2)Y_B f \bigr] = \\
  &= w'(n+2w'-2)w(n+2w-2)Y^B (\si')^{a_1 \ldots a_{k'}} 
  \na_{a_1} \ldots \na_{a_{k'}} \De^\ell Y_B f  = 0. 
\end{split}
\end{eqnarray}
The reason is that the tractor $Y^B$ contracts nontrivially
only with $X_B$ according to \nn{trm} and $X_B$ appear on the 
right hand side of $\na_{(a_1} \dots \na_{a_{k'})_0} \De^\ell Y_Bf$ 
according to \nn{nacomp} and \nn{Decomp} involves curvature.  
Analogously we obtain
\begin{equation} \label{YZ}
  w'(n+2w'-2) Y^B Q_{k',\ell}^{\si'} [(n+2w-2) Z_B^b \na_b f] = 0. 
\end{equation}
Looking at the $X_B$-terms of $Q_{k',\ell}^{\si'} (-X_B \De f)$, we see from
\nn{nacomp} and \nn{Decomp} that
\begin{eqnarray} \label{YX}
\begin{split}
  &w'(n+2w'-2) Y^B Q_{k',\ell}^{\si'} [-X_B \De f] = \\
  &=- w'(n+2w'-2) (\si')^{a_1 \ldots a_{k'}} \na_{a_1} \ldots \na_{a_{k'}} 
   \De^{\ell+1}f + lot_{\leq k'+2\ell+1}. 
\end{split}
\end{eqnarray}

Next we shall compute $(n+2w'-2)Z^{Bb} \na_b Q_{k',\ell}^{\si'} D_Bf$, 
we obtain again three summands. 
This is contraction of $(n+2w'-2)Z_B^b$ with
\begin{align*}
\na_b Q_{k',\ell}^{\si'} D_Bf =
&\bigl[ (\na_b (\si')^{a_1 \ldots a_{k'}}) \na_{a_1} \ldots \na_{a_{k'}}
 +(\si')^{a_1 \ldots a_{k'}} \na_b \na_{a_1} \ldots \na_{a_{k'}} \\
&+lot_{< k'} \bigr] \De^\ell D_B f .
\end{align*}
We need to discuss here only the first two terms
in the square bracket here and
only $Z_B^{\bar{b}}$-terms according to \nn{trm}. 
First, it is easy to see that
$$
(n+2w'-2)Z^{Bb}
(\na_b (\si')^{a_1 \ldots a_{k'}}) 
\na_{a_1} \ldots \na_{a_{k'}} \De^\ell D_B f = lot_{\leq k'+2\ell+1}.
$$
(The component $w(n+2w-2)Y_B$ of $D_B$ does not contribute to the right 
hand side of the previous display at all and the remaining components
$(n+2w-2)Z_B^{\bar{b}} \na_{\bar{b}}$ and $-X_B\De$ contribute by terms of 
the equal $\leq k'+2\ell+1$.) Hence it remains to collect
$Z_B^{\bar{b}}$-terms of 
$(\si')^{a_1 \ldots a_{k'}} \na_b \na_{a_1} \ldots \na_{a_{k'}} \De^\ell D_B f$.
Applying \nn{Dform} to $D_B$, we obtain three more summands. A short 
computation reveals that
\begin{align}
&(n+2w'-2)Z^{Bb} (\si')^{a_1 \ldots a_{k'}} \na_b \na_{a_1} \ldots \na_{a_{k'}} 
 \De^\ell \bigl[ w(n+2w-2)Y_B f \bigr]  = 0, \label{ZY} \\
&(n+2w'-2)Z^{Bb} (\si')^{a_1 \ldots a_{k'}} 
 \na_b \na_{a_{1'}} \ldots \na_{a_{k'}} 
 \De^\ell \bigl[ (n+2w-2)Z_B^{\bar{\,b}} \na_{\bar{b}} f \bigr] = 
 \label{ZZ}\\
&\quad =(n+2w'-2)(n+2w-2)Z^{Bb} (\si')^{a_1 \ldots a_{k'}} \na_b Z_B^{\bar{b}} 
 \na_{a_1} \ldots \na_{a_{k'}} \De^\ell \na_{\bar{b}} f = \notag \\
&\quad= (n+2w'-2)(n+2w-2) \si^{a_1 \ldots a_{k'}} \De^{\ell+1} f
 \notag \\
&(n+2w'-2) Z^{Bb} \si^{a_1 \ldots a_{k'}} \na_b \na_{a_1} \ldots \na_{a_{k'}} 
 \De^\ell \bigl[ -X_B \De f \bigr] =
 \label{ZX} \\
&\quad =-(n+2w'-2) Z^{Bb} \si^{a_1 \ldots a_{k'}} 
 \na_b \na_{a_1} \ldots \na_{a_{k'}}
 \bigl[ 2\ell Z_B^{\bar{b}} \na_{\bar{b}} \De^\ell +
 X_B \De^{\ell+1} \bigr] f = \notag \\
&\quad =-(n+2w'-2) Z^{Bb} (\si')^{a_1 \ldots a_{k'}} \na_b \bigl[
 X_B \na_{a_1} \ldots \na_{a_{k'}} \De^{\ell+1} \notag \\
&\qquad\qquad + Z_B^{\bar{b}} \bigl( 
 2\ell \na_{a_1} \ldots \na_{a_{k'}} \na_{\bar{b}} \De^\ell + 
 k'\bg_{\bar{b}a_1} \na_{a_2} \ldots \na_{a_{k'}} \De^{\ell+1} \bigl) \bigr] f 
 = \notag \\
&\quad = -(n+2w'-2) (2\ell+k'+n) (\si')^{a_1 \ldots a_{k'}} 
 \na_{a_1} \ldots \na_{a_{k'}} \De^{\ell+1} f. \notag
\end{align}
Beside the fact that $Z^{Bb}$ contracts nontrivially only with 
$Z_B^{\bar{b}}$,
we have used \nn{Decomp} to commute $\De^\ell$ with $Z_B^{\bar{b}}$,
\nn{nacomp} to commute $\na_{a_1} \ldots \na_{a_{k'}}$ with $Z_B^{\bar{b}}$ 
and \nn{XYZ} to commute $\na_b$ with $Z_B^{\bar{b}}$.

It remains to compute $-X^B \De Q_{k',\ell}^{\si'} D_Bf$. The computation
is analogous to previous cases but getting more tedious. 
First we observe
\begin{align*}
-X^B \De Q_{k',\ell}^{\si'} D_Bf 
=-X^B \bigl[ 
&(\De (\si')^{a_1 \ldots a_{k'}}) \na_{a_1} \ldots \na_{a_{k'}} 
 +2(\na^p (\si')^{a_1 \ldots a_{k'}}) \na_p \na_{a_1} \ldots \na_{a_{k'}} \\ 
&+(\si')^{a_1 \ldots a_{k'}} \na_{a_1} \ldots \na_{a_{k'}} \De 
 + lot_{\leq k'-1} \bigr] \De^\ell D_B f.
\end{align*}
We shall discuss only the first three terms in the square bracket here.
One can compute that
$$
-X^B \bigl[ (\De (\si')^{a_1 \ldots a_{k'}}) \na_{a_1} \ldots \na_{a_{k'}} 
+2(\na^p (\si')^{a_1 \ldots a_{k'}}) \na_p \na_{a_1} \ldots \na_{a_{k'}} \bigr] 
\De^\ell D_B f
= lot_{\leq k'+2\ell+1}
$$
so it remains to compute only
$-X^B (\si')^{a_1 \ldots a_{k'}} \na_{a_1} \ldots \na_{a_{k'}} 
\De^{\ell+1}D_B f$.
This yields three summands according to \nn{Dform}. After some computation
we obtain
\begin{align}
&-X^B (\si')^{a_1 \ldots a_{k'}} \na_{a_1} \ldots \na_{a_{k'}} 
 \De^{\ell+1} \bigl[ w(n+2w-2)Y_B f \bigr]  = \label{XY} \\
&\qquad = -w(n+2w-2) (\si')^{a_1 \ldots a_{k'}} \na_{a_1} \ldots \na_{a_{k'}} 
 \De^{\ell+1} f, \notag \\
&- X^B (\si')^{a_1 \ldots a_{'k}} \na_{a_1} \ldots \na_{a_{'k}} 
 \De^{\ell+1} \bigl[ (n+2w-2)Z_B^{\bar{\,b}} \na_{\bar{b}} f \bigr] = 
 \label{XZ}\\
&\qquad = -(n+2w-2) X^B (\si')^{a_1 \ldots a_{k'}} 
 \na_{a_1} \ldots \na_{a_{k'}} \notag \\
&\qquad\qquad \bigl[-2(\ell+1) Y^B \na^{\bar{\,b}} \De^\ell \na_{\bar{b}} +   
 Z_B^{\bar{\,b}} \na_{\bar{b}} \De^{\ell+1} \bigr] f = \notag \\
&\qquad= -(n+2w-2) \bigl[ -2(\ell+1) -k' \bigr]
 (\si')^{a_1 \ldots a_{k'}} \na_{a_1} \ldots \na_{a_{k'}} \De^{\ell+1} f,
 \notag \\
&- X^B (\si')^{a_1 \ldots a_{k'}} \na_{a_1} \ldots \na_{a_{k'}} 
 \De^{\ell+1} \bigl[ -X_B \De f \bigr] =
 X^B \si^{a_1 \ldots a_{k'}} \na_{a_1} \ldots \na_{a_{k'}}
 \label{XX} \\
&\qquad\qquad 
 \bigl[ -(\ell+1)(n+2\ell) Y_B \De^{\ell+1} 
 + 2(\ell+1) Z_B^{\,\bar{b}} \na_{\bar{b}} \De^{\ell+1}
 + X_B \De^{\ell+2} \bigr] f = \notag \\
&\qquad= \bigl[ -(\ell+1)(n+2\ell) - 2k'(\ell+1)  \bigr] 
 (\si')^{a_1 \ldots a_{k'}} \na_{a_1} \ldots \na_{a_{k'}} \De^{\ell+1} f. 
 \notag
\end{align}

The last step of the proof is to sum up the right hand sides of 9
relations \nn{YY}, \nn{YZ}, \nn{YX}, \nn{ZY}, \nn{ZZ}, \nn{ZX} and
\nn{XY}, \nn{XZ}, \nn{XX} above. That is, we need to compute the scalar
\begin{gather*}
-w'(n+2w'-2) + (n+2w'-2)(n+2w-2) - (n+2w'-2) (2\ell+k'+1) \\
-w(n+2w-2) + (n+2w-2)(2\ell+k'+2) - (\ell+1)(n+2\ell+2k')
\end{gather*}
where $w' = w+\de'-2\ell-1$. This requires some work, the result is
$-(\de'-\ell)(n+2\de'+2k'-2\ell-2)$ and the proposition follows.
Note the resulting scalar does not depend on $w$; this is a good 
verification that the computations throughout the proof are correct.
\end{proof}

\begin{theorem} \label{quant}
Let $k',\ell \geq 0$, 
$(\si')^{a_1 \ldots a_{k'}} \in \cE^{(a_1 \ldots a_{k'})_0}[\de']$
and $f \in \cE[w]$, $\de',w \in \bbR$.
Then
$$ Q_{k',\ell}^{\si'} := D^{B_1} \cdots D^{B_{k'}} Q_{k',0}^{\si'} 
D_{B_{k'}} \cdots D_{B_1}: \cE[w] \to \cE[w+\de'-2\ell] $$ 
defines the conformally invariant quantization with 
the leading term $(\si')^{a_1 \ldots a_{k'}}\na_{a_1} \ldots \na_{a_{k'}} 
\De^\ell$ (up to a sign) for every weight $\de'$ satisfying
\begin{equation} \label{res}
\de' \not\in \Si_{k',\ell} := 
\Si_{k',0} \cup \Si_{k',\ell}' \cup \Si_{k',\ell}''
\end{equation}
where $\Si_{k',0}$ is given by \nn{crit0}, 
\begin{equation}
\Si_{k',\ell}' = \{ (j-1) \mid j = 1,\ldots,\ell \}, \quad
\Si_{k',\ell}'' = \{ -\frac{1}{2} (n+2k'-2j) \mid j = 1,\ldots,\ell \}
\quad \mbox{for} \  \ell \geq 1. 
\end{equation}
We put $\Si_{k',0}' = \Si_{k',0}'' := \emptyset$.
Moreover, $Q_{k',\ell}^{\si'}$ is strongly invariant
in the sense of Theorem \ref{quant0}. 
\end{theorem}

\begin{proof}
The set of critical weights $\Si_{k',\ell}$ easily follows 
(by induction with respect to $\ell$) from
Proposition \ref{ind}. Since the tractor $D$ and $Q_{k',0}^{\si'}$ are
strongly invariant, the last claim is obvious.
\end{proof}

\begin{remark*}
Let us note the previous theorem yields an inductive formula for the 
conformal quantization as 
$Q_{k',\ell+1}^{\si'} = D^B Q_{k',\ell}^{\si'} D_B$. Similarly, we can 
describe the set of critical weights inductively as 
$\Si_{k',\ell+1} = 
\Si_{k',\ell} \cup \{ \ell, -\frac{1}{2} (n+2k'-2\ell-2)\}$
where $\Si_{k',0}$ is given by \nn{crit0}.
\end{remark*}

\section{Critical weights}
\label{scrit}

We shall discuss the cases $\de' \in \Si_{k',\ell}$ from \nn{res} in detail. 
First, a simple calculation shows

\begin{lemma} \label{disjoint}
(i) $2\ell \not\in \Si_{k',\ell}$ for all $k',\ell \geq 0$. \\
(ii) The sets $\Si_{k',0}$ and $\Si_{k',\ell}' \cup \Si_{k',\ell}''$ are 
disjoint. \qed
\end{lemma}

The symbols of the quantization $\cE[w] \to \cE[w]$ (i.e.\ with 
zero shift) are of a special interest \cite{DLOex}. The flat quantization
developed there is never critical for such symbols \cite[3.1]{DLOex}.
The previous lemma (i) recovers this observation for the curved quantization
$Q_{k',\ell}^{\si'}$.

The critical weights are closely related to existence to natural 
linear conformal 
operators. They are completely classified in the locally flat case
\cite[(3.1)]{BoCo} (or see the summary in \cite[Section 3]{EaSl}).
Using this we obtain

\begin{proposition} \label{critop}
Assume the manifold $M$ is conformally flat. 
If $\de' \in \Si_{k',\ell}$ then there exists a nontrivial
natural linear conformal operator on $\cE^{(a_1\ldots a_{k'})_0}[\de']$
as follows
\begin{align*}
& \cE^{(a_1 \ldots a_{k'})_0}[\de'] \longrightarrow 
\cE^{(a_1 \ldots a_{i-1})_0}[\de'], && \de = -(n+k'+i-2) \in \Si_{k',0}, \\
& \cE^{(a_1 \ldots a_{k'})_0}[\de'] \longrightarrow 
\cE^{(a_1 \ldots a_{{k'}+j})_0}[\de'-2j], && \de' = j-1 \in \Si_{{k'},\ell}', 
\\
& \cE^{(a_1 \ldots a_{k'})_0}[\de'] \longrightarrow 
\cE^{(a_1 \ldots a_{k'})_0}[\de'-2j], && \de' = -\frac{1}{2}(n+2k'-2j) \in 
\Si_{k',\ell}''.
\end{align*}
\end{proposition}

The case $\de' \in \Si_{k',0}$ is a divergence type operator of the 
order $k'-i+1$, $\de' \in \Si_{k',\ell}'$ is the conformal 
Killing operator of the order $j$ and $\de' \in \Si_{k',\ell}''$ yields 
a Laplacian type operator of the order $2j$. Note this operator is not 
unique as generally $\Si_{k',\ell}' \cap \Si_{k',\ell}'' \not= \emptyset$.

\vspace{1ex}

The operator $Q_{k',\ell}^{\si'}$ does not provide a conformally invariant 
quantization for $\de' \in \Si_{k',\ell}$. Such a quantization can 
exists, though, for certain $w$, as observed in lower order cases
\cite{DuOvHam,Dj3}. (Note it is not unique even in the flat case then.)
Assuming $\de' \in \Si_{k',\ell}$, we shall find such $w$ for all $k',\ell$
in the flat setting; the curved case is more 
involved. In particular, it is closely related to existence of natural 
linear conformal operators
\begin{align*}
&S_p: \cE[p-1] \longrightarrow \cE_{(a_1 \ldots a_p)_0}[p-1], \quad
&&L_p: \cE[-n/2+p] \longrightarrow \cE[-n/2-p],\\
&S_p(f) = \na_{(a_1} \ldots \na_{a_p)_0}f + lot,
&&L_p(f) = \De^p f + lot,
\end{align*}
for $p \geq 1$ (so $p$ is \idx{not} an abstract index here).
If $n$ is odd or $M$ is conformally flat,
these operators exist for all $p \geq 1$.
In the curved case for $n$ even, 
$S_p$ exists for all $p \geq 1$ and $L_p$ exists for $1 \leq p \leq n$, 
see \cite{CSS,GJMS,GoHiNex}. They are strongly invariant 
(can be given by a strongly invariant formula) in the flat case;
in the curved case, $S_p$ is strongly invariant always and
$L_p$ only for $p<n$.

\begin{theorem} \label{critquant}
Assume $\de' \in \Si_{k',\ell}$ and $f \in \cE[w]$. 
Then there is always a choice of $w \in \bbR$ for which there is
a quantization $Q_{k',\ell}^{\si'}: \cE[w] \to \cE[w+\de]$ 
with the leading term 
$(\si')^{a_1 \ldots a_{k'}} \na_{a_1} \ldots \na_{a_{k'}} \De^\ell f$ 
in the flat case.
This is true also on curved manifols under an additional assumption
$\ell \leq n$. 

Explicitly, the quantization is given by formulae
\begin{align*}
&Q_{k',0}^{\si'} L_\ell: \cE[-n/2+\ell] \to \cE[\de'-n/2-\ell],
  &&\de' \in \Si_{k',\ell}' \cup \Si_{k',\ell}'' \\
&D^{B_1} \cdots D^{B_\ell} \io(\si') S_{k'} D_{B_1} \cdots D_{B_\ell}: 
  \cE[k'+\ell-1] \to \cE[\de'+k'-\ell-1],
  &&\de' \in \Si_{k',0}
\end{align*}
where $\io(\si')$ is the complete contraction of the image of $S_p$ 
with $\si'$.
\end{theorem}

\begin{proof}
The conformal invariance is obvious (recall $S_{k'}$ has the source space 
$\cE[k'-1]$ and is strongly invariant). It remains to verify the displayed
operators have 
the required leading term (up to a nonzero multiple). In the case
$\de' \in \Si_{k',\ell}' \cup \Si_{k',\ell}$, this follows from the leading 
term of $L_\ell$, properties of $Q_{k',0}^{\si'}$ in Theorem \ref{quant} and 
Lemma \ref{disjoint} (ii).

Assume $\de' \in \Si_{k',0}$ and denote by  $\ol{Q}_{k',\ell}^{\si'}$ 
the displayed operator for such $\de'$. We need to compute the leading 
term of $\ol{Q}_{k',\ell}^{\si'}$. 
Observe the generic quantization $Q_{k',\ell}^{\si'}$ is constructed
in a similar way as $\ol{Q}_{k',\ell}^{\si'}$ 
-- only the subfactor $Q_{k',0}^{\si'}$ of $Q_{k',\ell}^{\si'}$
(see the display in Theorem \ref{quant}) 
is replaced by $\io(\si')S_{k'}$ in $\ol{Q}_{k',\ell}^{\si'}$. 
It is mentioned in the proof of Proposition \ref{ind} that only the term 
$(\si')^{a_1 \ldots a_{k'}} \na_{a_1} \ldots \na_{a_{k'}} \De^\ell$  of 
$Q_{k',\ell}^{\si'}$ contributes to the generic leading term
$(\si')^{a_1 \ldots a_{k'}} \na_{a_1} \ldots \na_{a_{k'}} \De^{\ell+1}$  
of $\wt{Q}_{{k'},\ell}^{\si'}$, see Proposition \ref{ind} for the notation.
However $\io(\si')S_{k'}$ has the leading term
$(\si')^{a_1 \ldots a_{k'}} \na_{a_1} \ldots \na_{a_{k'}}$ 
for $\de' \in \Si_{k',0}$ as well as 
$Q_{k',0}^{\si'}$ for $\de' \not\in \Si_{k',0}$. It follows
that $\ol{Q}_{k',\ell}^{\si'}$ has 
$(\si')^{a_1 \ldots a_{k'}} \na_{a_1} \ldots \na_{a_{k'}} \De^\ell$ 
as the leading term
for all $\de' \in \Si_{k',\ell}' \cup \Si_{k',\ell}''$. Using
Lemma \ref{disjoint}(ii) the theorem follows.
\end{proof}

\noindent
\begin{remark*}
As expected, the quantization used in the previous Proposition is not 
unique. If, for example,  
$\de' \in \Si_{k',\ell}' \cup \Si_{k',\ell}''$ but
$\de' \not\in \Si_{k',\ell_0}' \cup \Si_{k',\ell_0}''$ for some 
$\ell_0 < \ell$,
one can use also the operator $Q_{k',\ell_0}^{\si'} L_{\ell-\ell_0}$
which is invariant on $\cE[-n/2+\ell-\ell_0]$. (A similar idea can be used
for $\de' \in \Si_{k',0}$.)
\end{remark*}

\section{Comparision with related results}

There are several related results concerning conformal quantization 
either in the flat case \cite{DLOex} or in lower order curved cases
\cite{DuOvHam,Dj3,MaRaNc}. On the other hand, conformal quantization 
is a special case of bilinear operators constructed in \cite{KrTh}.
We discuss \cite{KrTh} and \cite{DLOex} in more details here. 

In the groundbreaking Kroeske's thesis \cite{KrTh}, a general construction
of invariant bilinear operators (or ``invariant pairing'') for (curved) 
parabolic geometries is developed. However, we require that
$Q_{k',\ell}^{\si'}: \cE[w] \to \cE[w+\de']$ is defined for every 
$w \in \bbR$ whereas the construction in \cite{KrTh} generally yields
\idx{couples} of critical weights $(\de',w)$.
Considering the conformal case, it is probably possible to obtain 
the quantization on densities from the detailed exposition in 
\cite[Section 5]{KrTh} with some set of critical weights.
Our construction of the  bilinear operator $Q$ is much simpler as it is 
designed to the special case needed here.

In \cite{DLOex}, the study of conformal quantization was initiated.
The set of critical weights (as a subset of ``resonant'' weights)  
agrees with our result up
to the order 2. In the order 3, also the critical case
$\cE^{(a_1a_2a_3)}[\de]$
where $\de= -\frac{2}{3}(n+2)$. (Note we have used
$\cE[-nw] = \cF_{w}$ (cf.\ the introduction) to pass to our notation.
In fact, the value
$\frac{2(n+2)}{3n}$ is obtained by the choice $(k,l,s,t) = (3,0,1,0)$ 
in \cite[(3.7)]{DLOex}, see also \cite[Theorem 3.5,3.6]{DLOex}.)
$\cE^{(a_1a_2a_3)}[\de]$ has two irreducible components, in particular
$\cE^{(a_1a_2a_3)_0}[\de]$ and $\cE^a[\de+2]$. However
$\de \not\in \Si_{3,0}$ and $\de+2 \not\in \Si_{1,1}$ for 
$\de= -\frac{2}{3}(n+2)$ and generic $n$. 
Note there is no nontrivial natural linear flat conformal operator on 
$\cE^a[\de+2]$ or $\cE^{(a_1a_2a_3)_0}[\de]$ in generic dimensions.

\vspace{1ex}

It seems plausible the critical set $\Si_{k',\ell}$  is minimal for conformal 
quantization with the corresponding leading term. Although no 
non-existence results for higher orders are known (up to our knowledge), 
we conjecture that
if $(\si')^{a_1 \ldots a_{k'}} \in \cE^{(a_1 \ldots a_{k'})_0}[\de']$,
$\de' \in \Si_{k',\ell}$ then there is no conformal quantization 
$\cE[w] \to \cE[w+\de'-2\ell]$  with the leading 
term $\si^{a_1 \ldots a_{k'}} \na_{a_1} \ldots \na_{a_{k'}} \De^\ell$ 
for generic $w \in \bbR$. The minimality is closely related
to Proposition \ref{critop} and Theorem \ref{critquant}. In particular, 
we expect a version of Proposition \ref{critop} to be a necessary condition 
for nonexistence.

\section{Examples}

The tractor formulae are easily rewritten to the usual formulae
in the Levi--Civita covariant derivative and its curvature.
We shall demonstrate this on the quantization of the order three
(which in fact known \cite{Dj3}). There are two irreducible
leading terms:

\begin{example}
We shall start with the case $(\si')^{abc}\na_a\na_b\na_c$ where
$(\si')^{abc} \in \cE^{(abc)_0}[\de']$. We shall avoid the general result
from \cite{KrTh} as one can directly verify the differential operators
\begin{align*}
(\si')^{abc} \mapsto M^{ABC}_{\ a\,b\,c} & (\si')^{abc} :=
(n+\de'+2)(n+\de'+3)(n+\de'+4) Z^A_{\,a} Z^B_{\,b} Z^C_{\,c} (\si')^{abc} \\
&-3 (n+\de'+2)(n+\de'+3) X^{(A} Z^B_{\,b} Z^{C)}_{\,c} \na_p(\si')^{pbc} \\
&+3 (n+\de'+3) X^{(A} X^B Z^{C)}_{\,c} 
\bigl( \na_p \na_q + (n+\de'+4) P_{pq} \bigr) (\si')^{pqc} \\
&-X^A X^B X^C \bigl[ \na_p \bigl( \na_q \na_r + (n+\de'+4) P_{qr} \bigr) 
+ 2(n+\de'+3) P_{pq} \na_r \bigr] (\si')^{pqr}
\end{align*}
and 
\begin{align*}
f \mapsto \wt{D}_{ABC} f :=  
& w(w-1)(w-2) Y_A Y_B Y_C f \\
&+ 3 (w-1)(w-2) Y_{(A} Y_B Z_C^{\,c} \na_c f \\
&+ 3 (w-2) Y_{(A} Z_B^{(b} Z_{C)}^{\,c)_0} \bigl( \na_b\na_c + wP_{bc} \bigr) \\
&+ Z_{(A}^{(a} Z_B^b Z_{C)}^{\,c)_0} \bigl[
\na_a (\na_b\na_c + wP_{bc}) + 2(w-1) P_{bc} \na_a \bigr] f
\end{align*}
where $f \in \cE[w]$. One easily verifies directly they are conformally 
invariant. (Note the $M$ is a special case of the ``middle operator'' from
\cite[2.1.4]{SiTh}.). 

The target space of the operator 
$M^{ABC}_{\ a\,b\,c}$ is a subbundle of $\cE^{(ABC)}[\de'+3]$ and  
target space of $\wt{D}_{ABC}$ is a quotient of 
$\cE_{(ABC)}[w-3]$. These two target spaces are dual to each other via 
the tractor metric $h$. Thus the contraction
\begin{align*}
f \mapsto \bigl( M^{ABC}_{\ a\,b\,c}  (\si')^{abc} \bigr) \wt{D}_{ABC} f = 
&(n+\de'+2)(n+\de'+3)(n+\de'+4) \\
&(\si')^{abc} [\na_a (\na_b\na_c + wP_{bc}) + 2(w-1) P_{bc} \na_a]f + lot
\end{align*}
where $lot$ means ``lower order terms'', is conformally invariant. 
Note the we see directly from \nn{trm} which terms of
$M^{ABC}_{\ a\,b\,c} (\si')^{abc}$ and $\wt{D}_{ABC} f$ contract nontrivially
with each other.
\end{example}

\begin{example}
Another possible leading term in the third order is
$(\si')^b \na_b \De$, $(\si')^b \in \cE^b[\de']$. Also this case can be solved 
without the general formula in Theorem \ref{quant}. Similarly as in 
the previous example, we shall start with the conformally invariant operator
$(\si')^b \mapsto M^B_{\;b}(\si')^b := 
(n+\de') Z^B_{\;b}(\si')^b - X^A\na_p (\si')^p$.
Next we apply the operator $D_A$ and symmetrize over the 
tractor indices. Overall we obtain
\begin{align*}
(\si')^b \mapsto & M^B_{\;b}(\si')^b \mapsto
D^{(A} M^{B)}_{\;b}(\si')^b = \de'(n+\de')(n+2\de') Y^{(A} Z^{B)}_{\,b} 
(\si')^b \\
&+(n+2\de') Z^{(A}_{\ a} Z^{B)}_{\;b} \bigl[ 
(n+\de') \na^{(a}(\si')^{b)_0} 
+\frac{1}{n} \de' \bg^{ab} \na_p (\si')^p \bigr] \\
&-\de'(n+2\de') X^{(A} Y^{B)} \na_p (\si')^p \\
&-X^{(A} Z^{B)}_{\;b} \bigl[ 
(n+2\de'-2) \bigl( \na^b \na_p + (n+\de') P^b{}_p \bigr) (\si')^p
+(n+\de') \bigl( \De + (\de'+1)J \bigr) (\si')^b \bigr] \\
&+X^{(A} X^{B)} \bigl[ \bigl( \De+ \de' J \bigr) \na_p 
+ (n+\de') \bigl( (\na_p J) + 2 P_{pq} \na^q \bigr) \bigr] (\si')^p
\end{align*}
after some computation using \nn{Dform}. The target space of 
$D^{(A} M^{B)}_{\;b}$ is a subbundle of $\cE^{(AB)}[\de']$. Hence
we need an operator which
takes $f \in \cE[w]$ in to the dual of this target space. Naively, we can 
use $D_AD_Bf$ but this would kill the leading term $(\si')^p \na_p\De$ for
$w = -\frac{n}{2}+2$. In fact, the dual (up to a conformal weight) 
to $D^{(A} M^{B)}_{\;b}(\si')^b$ 
is a quotient of $D_AD_B f$ which we denote by $\wt{T}_{AB}f$. 
After some computation, we obtain that
\begin{align*}
f \mapsto &\wt{T}_{AB}f := w(w-1)(n+2w-2) Y_A Y_B f + \\
&+2(w-1)(n+2w-2) Y_{(A} Z_{B)}^{\;b} \na_b f \\
&+Z_{(A}^{\ a} Z_{B)}^{\;b} \bigl[
(n+2w-2) \bigl( \na_{(a} \na_{b)_0} +w P_{(ab)_0} \bigr) 
+\frac{2}{n}(w-1) \bg_{ab} \bigl( \De + wJ \bigr) \bigr] f\\
&-2(w-1) X_{(A} Y_{B)} \bigl( \De + wJ \bigr) f \\
&-2 X_{(A} Z_{B)}^{\;b} \bigl[ \na_b \bigl( \De + wJ \bigr)
+(n+2w-2)P_b{}^p\na_p \bigr] f 
\end{align*}
is conformally invariant. Summarizing, we obtain the conformally invariant 
quantization on $\cE[w]$ by
\begin{align*}
f \mapsto &\bigl( D^{(A} M^{B)}_{\;b}(\si')^b \bigr) T_{AB} f = \\
&=-\de'(n+\de')(n+2\de') (\si')^b
\bigl[\na_b \bigl( \De + wJ \bigr) +(n+2w-2)P_b{}^p\na_p \bigr] f + lot
\end{align*}
where ``$lot$'' denotes lower order terms.
\end{example}

\end{document}